\documentclass{amsart}

\usepackage[T1]{fontenc}
\usepackage[dvipsnames]{xcolor}
\usepackage{amssymb}
\usepackage{amsthm}
\usepackage{mathrsfs}
\usepackage{bigints}
\usepackage{mathtools}
\DeclareMathOperator*{\supp}{supp}
\DeclareMathOperator*{\sgn}{sgn}
\DeclareMathOperator*{\diam}{diam}

\usepackage{cases}
\usepackage[shortlabels]{enumitem}
\newtheorem{theorem}{Theorem}[section]
\newtheorem{definition}[theorem]{Definition}
\newtheorem*{definition*}{Definition}
\newtheorem*{remark*}{Remark}
\newtheorem*{problem*}{Problem}
\newtheorem{lemma}[theorem]{Lemma}
\newtheorem{proposition}[theorem]{Proposition}
\newtheorem{question}{Question}
\newtheorem{problem}{Problem}
\newtheorem{corollary}[theorem]{Corollary}
\newtheorem{remark}[theorem]{Remark}

\usepackage{hyperref}

\begin{document}
\title[Frames of translates in~ $L_p(\mathbb{R}^d)$]{Unconditional Schauder frames of translates in~ $L_p(\mathbb{R}^d)$}
\date{}
\author{Miguel Berasategui \and Daniel Carando}

\address{Departamento de Matem\'{a}tica - Pab I,
	Facultad de Cs. Exactas y Naturales, Universidad de Buenos Aires,
	(1428) Buenos Aires, Argentina, and CONICET-IMAS}

\thanks{This work was partially supported by CONICET-PIP 11220130100329CO, ANPCyT PICT 2015-2299}


\maketitle

\begin{abstract}
We show that, for $1<p \le 2$, the space $L_p(\mathbb{R}^d)$ does not admit unconditional Schauder frames $\left\lbrace f_i,f_i'\right\rbrace_{i\in\mathbb{N}}$ where $\left\lbrace f_i\right\rbrace$ is a sequence of translates of finitely many functions and $\left\lbrace f_i'\right\rbrace$ is seminormalized. In fact, the only subspaces of $L_p(\mathbb{R}^d)$ admitting such Banach frames are those isomorphic to $\ell_p$.  On the other hand, if $2<p<+\infty$ and $\left\lbrace \lambda_{i}\right\rbrace_{i\in\mathbb{N}}\subseteq \mathbb{R}^d$ is an unbounded sequence, there is a subsequence $\left\lbrace \lambda_{m_i}\right\rbrace_{i\in\mathbb{N}}$, a function $f\in L_p(\mathbb{R}^d)$, and a seminormalized sequence of bounded functionals $\left\lbrace f_i'\right\rbrace_{i\in\mathbb{N}}$ such that $\left\lbrace T_{\lambda_{m_i}}f,f_i'\right\rbrace_{i\in\mathbb{N}}$ is an unconditional Schauder frame for $L_p(\mathbb{R}^d)$.
\end{abstract}

\section{Introduction.}\label{sec-intro} We study Schauder frames of translates in subspaces of $L_p(\mathbb R^d).$ For $\lambda\in \mathbb{R}^d$, the \emph{translation operator} $T_{\lambda}$ is defined by
$$
\left(T_{\lambda}f\right)(x)=f(x-\lambda),
$$
where $f$ is a function defined in $\mathbb R^d$. Given a sequence $\{\lambda_i\}_i\subset \mathbb R^d$ and a function $f\in L_p(\mathbb R^d)$, the closed linear span of the sequence of translates $$X=\overline{\left[T_{\lambda_i} f:i\in\mathbb{N}\right]}\subset L_p(\mathbb R^d)$$ has been studied by many authors (see, for example,  \cite{Atzmon1996, Bruna2006, Christensen1995,Christensen1999,Olevskii2018,Olson1992,Wiener1933}, where techniques from harmonic analysis are the main tools, and \cite{Freeman2014, Odell2011}, where Banach space techniques are essential). Systems formed by translates of finitely many functions were also studied (see, for example,  \cite{Liu2012}).

One might classify many of the different questions which are behind this research into two families. The first one focuses on $X$ itself: does it coincide with $L_p(\mathbb R^d)$? If not, what can we say about $X$? (do we have an isomorphic description? is it complemented in $L_p(\mathbb R^d)$? The second family of questions focuses on the sequence $\{T_{\lambda_i} f\}_i$: is it a Schauder basis of $X$? or a Schauder frame? What about unconditionality? However, the (maybe) most natural question belongs to the intersection of these families: does $L_p(\mathbb R^d)$ admits (unconditional) Schauder bases (or frames) formed by translates? In other words, can we have both $X=L_p(\mathbb R^d)$ and $\{T_{\lambda_i} f\}_i$ a Schauder basis/frame? We recall some results in this direction that motivated our research.
\begin{theorem}\cite[Theorem 2]{Olson1992}
If $\{ T_{\lambda_i} f \}_{i\in \mathbb{N}}\subseteq L_2(\mathbb{R})$ is an unconditional basic sequence, then $\overline{\left[ T_{\lambda_i}f :i\in\mathbb{N}\right]}\ne L_2(\mathbb{R})$.
\end{theorem}
\begin{theorem}\cite[Corollary 2.10]{Odell2011}\label{teoremabaseequivalente12}
Let $1\le p \le 2$ If $\left\lbrace T_{\lambda_i} f\right\rbrace_{i\in \mathbb{N}}\subseteq L_p(\mathbb{R})$ is an unconditional basic sequence, it is equivalent to the unit vector basis of $\ell_p$. Hence, if $1\le p <2$, then $\overline{\left[ T_{\lambda_i}f:i\in\mathbb{N}\right]}\not=L_p(\mathbb{R})$.
\end{theorem}

\begin{theorem}\cite[Theorem 4.3]{Liu2012}
Let $1<p\le 2$. If $\{f_i\}_{i\in \mathbb{N}}\subseteq L_p(\mathbb{R}^d)$ is an unconditional basic sequence of translates of elements of a finite set, then $\overline{\left[f_i:i\in\mathbb{N}\right]}\not=L_p(\mathbb{R}^d)$.
\end{theorem}
\begin{theorem}\cite[Theorem 2.1, Corollary 2.3]{Freeman2014}\label{teoremabases2<p<+infinito}
Let $2<p <+\infty$, and let $\left\lbrace T_{\lambda_i}f \right\rbrace_{i\in \mathbb{N}}\subseteq L_p(\mathbb{R})$ be an unconditional basic sequence. If $\overline{\left[T_{\lambda_i}f:i\in\mathbb{N}\right]}$ is complemented in $L_p(\mathbb{R}^d)$, $\left\lbrace T_{\lambda_i}f \right\rbrace_{i\in \mathbb{N}}$ is equivalent to the unit vector basis of $\ell_p$. Hence, $\overline{\left[T_{\lambda_i} f:i\in\mathbb{N}\right]}\not=L_p(\mathbb{R}^d)$.
\end{theorem}
\begin{theorem}\label{teoremamarcoparap>2}\cite[Theorem 3.2]{Freeman2014}
Let $2<p<+\infty$, and $d\in\mathbb{N}$. If $\left\lbrace \lambda_i\right\rbrace_{i\in\mathbb{N}}\subseteq \mathbb{R}^d$ is an unbounded sequence, there exists $f\in L_p(\mathbb{R}^d)$ and $\left\lbrace f_i'\right\rbrace_{i\in\mathbb{N}} \subseteq \left(L_{p}(\mathbb{R}^d)\right)'$ such that $\left\lbrace T_{\lambda_i}f,f_i'\right\rbrace_{i\in\mathbb{N}}$ is an unconditional Schauder frame for $L_p(\mathbb{R}^d)$.
\end{theorem}
The proof of Theorem \ref{teoremamarcoparap>2} given in \cite{Freeman2014} gives  a Schauder frame $\left\lbrace T_{\lambda_i}f,f_i'\right\rbrace_{i\in\mathbb{N}}$ for $L_p(\mathbb{R}^d)$ such that $\{f_i'\}_{i\in\mathbb{N}}$ tends to zero in norm as $i$ tends to infinity. Identifying the dual space $\left(L_p(\mathbb{R}^d\right)'$ of $L_p(\mathbb{R}^d)$ with $L_{p'}(\mathbb{R}^d)$ in the usual way $\left(\frac{1}{p}+\frac{1}{p'}=1\right)$, the authors ask the following question:
\begin{problem}\cite[Problem 6.3]{Freeman2014}\label{pregunta1}
Let $1<p<+\infty$, and let $\{f_i\}_{i\in\mathbb{N}}$ be a sequence of translates of $f\in L_p(\mathbb{R})$. Is there a seminormalized sequence $\{f_i'\}_{i\in\mathbb{N}}\subseteq L_{p'}\left(\mathbb{R}\right)$ such that $\left\lbrace f_i,f_i'\right\rbrace$ forms an unconditional Schauder frame for $L_p(\mathbb{R})$?
\end{problem}

In this article, we study mainly unconditional (approximate) Schauder frames of translates of a function or of finitely many functions for subspaces of $L_p(\mathbb{R}^d)$, and in particular, we focus on Problem \ref{pregunta1}.
In Section \ref{preliminares}, we introduce some notation, recall some known facts and prove some general results. We study Schauder frames and approximate Schauder frames in general Banach spaces, and present sufficient conditions for the existence of seminormalized coordinates, which we will use in the proofs of our main results. In Section \ref{generalresults}, we introduce Schauder frames of translates in $L_p(\mathbb R^d)$ and prove some technical results that are used in the sequel. Section~\ref{results_in_L_p} has the main results of the article. In Section~\ref{2<p} we study the case $2<p<+\infty$. Modifying the proof of Theorem \ref{teoremamarcoparap>2} from \cite{Freeman2014}, we  prove in Theorem~\ref{Teoremamarcosincondicionalescon2<p<+infinito} that, for every unbounded sequence $\left\lbrace \lambda_i\right\rbrace_{i\in\mathbb{N}}\subseteq \mathbb{R}^d$, we can take  a subsequence $\left\lbrace \lambda_{m_i}\right\rbrace_{i\in\mathbb{N}}\subseteq \mathbb{R}^d$, a function $f\in L_p(\mathbb{R}^d)$, and a seminormalized sequence $\{f_i'\}_{i\in\mathbb{N}}\subseteq L_{p'}(\mathbb{R}^d)$ such that $\left\lbrace T_{\lambda_{m_i}}f,f_i'\right\rbrace_{i\in\mathbb{N}}$ is an unconditional Schauder frame for $L_p(\mathbb{R}^d)$.
This shows, in particular, that there are sequences of translates $\{f_i\}$ for which the answer to the question in Problem \ref{pregunta1} is positive (for $2<p<+\infty$). In Section \ref{1<p<=2} we study the case $1<p\le 2$. In opposition to the previous case, we prove that for any such $p$, Problem \ref{pregunta1} has a negative answer, no matter how we chose the sequence of translates. In Section~\ref{p=1} we consider $p=1$. We show in Proposition~\ref{propositionele1} that, under a natural geometric condition on the sequences of translates, any subspace of $L_1(\mathbb{R}^d)$ admitting  an unconditional frame of translates must be isomorphic to $\ell_1$.

\section{Preliminaries.}\label{preliminares} \label{seminormalizedcoordinates}

In this section, we present some  definitions and basic results. Unless otherwise stated, we assume that all of the spaces we consider are infinite-dimensional and separable Banach spaces over $\mathbb{K}$, where $\mathbb{K}$ may be chosen to be either $\mathbb{C}$ or $\mathbb{R}$.
If $X$ is a Banach space, we denote its dual  by $X'$, and by $B_X$ its closed unit ball. If $T:X\to Y$ is a bounded operator, its transpose $T^*:Y'\to X'$ is  given by $T^*(y')=y'\circ T$ for $y'\in Y'$. For $\lambda\in \mathbb R^d$, $|\lambda|$ denotes the Euclidean norm of $\lambda$.

We say that the sequence $\{f_i\}_{i\in\mathbb{N}}\subset X$ is \emph{bounded below} if there is $r>0$ such that $||f_i||\ge r$ for all $i\in\mathbb{N}$. The sequence $\{f_i\}_{i\in\mathbb{N}}$ is called \emph{seminormalized} if it is both bounded and bounded below.

Recall that a Banach space $X$ has \emph{type p} if there exists $M>0$ such that
\begin{align}
\left(\int_{0}^{1}\left|\left|\sum\limits_{i=1}^{n}r_i(t)f_i\right|\right|^{2}dt\right)^{\frac{1}{2}}\le& M\left(\sum\limits_{i=1}^{n}\left|\left|f_i\right|\right|^{p}\right)^{\frac{1}{p}}  \nonumber
\end{align}
 for every finite sequence $ \{f_i\}_{1\le i\le n}\subseteq X$.
Also, $X$ has \emph{cotype $s$}  if there is $M>0$ such that
\begin{align}
\left(\sum\limits_{i=1}^{n}\left|\left|f_i\right|\right|^{s}\right)^{\frac{1}{s}}
\le& M\left(\int_{0}^{1}\left|\left|\sum\limits_{i=1}^{n}r_i(t)f_i\right|\right|^{2}dt\right)^{\frac{1}{2}}\;\;\;\;\forall n\in\mathbb{N} \label{cotipo1}
\end{align}
for every $ \{f_i\}_{1\le i\le n}\subseteq X$. Here, $\left\lbrace r_i\right\rbrace_{i\in\mathbb{N}}$ is the sequence of Rademacher functions given by 	
\begin{align*}
r_i(t)=\sgn{(\sin(2^i\pi t))} \;\;\;\;\forall t\in[0,1] \;\;\;\forall i\in\mathbb{N}.
\end{align*}

It is known that for every measure space $(\Omega, \Sigma, \nu)$ and $1\le p<+\infty$, the space $L_p(\nu)$  has type $\min\{2,p\}$ and cotype $\max\{2,p\}$ \cite{Lindenstrauss1977.2}, and that if a space $X$ is of type $p>1$, its dual space $X'$ is of cotype $p'$, where $\frac{1}{p}+\frac{1}{p'}=1$ \cite[Proposition 1.e.17]{Lindenstrauss1977.2}.
From these facts and Orlicz's Theorem (see  \cite{Orlicz1933, Orlicz1933b} or Theorems 4.2.1 and 4.2.2 from \cite{Kadets1997}), we have the following.
\begin{proposition}\label{TeoremadeOrliczextendido}
Let $1<p<+\infty$, $s=\max\{2,p\}$,  $q=\max\{2,p'\}$and let  $X$ be a subspace $L_p(\nu)$, then
\begin{enumerate}[label=\rm{(\alph*)}]
\item \label{Orliczcomun} If $\sum\limits_{i=1}^{\infty}f_i$ converges unconditionally in $X$, then $\sum\limits_{i=1}^{\infty}||f_i||_X^{s}$ converges.
\item \label{Orliczdual} If $\sum\limits_{i=1}^{\infty}f_i'$ converges unconditionally in $X'$, then $\sum\limits_{i=1}^{\infty}||f_i'||_{X'}^{q}$ converges.
\end{enumerate}
\end{proposition}

The following remark is a consequence of the previous proposition and a standard uniform boundedness argument.

\begin{remark}\label{lemaacodadoabajovaaeleq} \label{lemaacodadoabajovaaeleq2}
With $p$, $q$, $s$ and $X$ as in the previous proposition, the following hold.
\begin{enumerate}[label=\rm{(\alph*)}]
\item
 If $\{f_i\}_{i\in\mathbb{N}}\subset X$ is bounded below and $\sum\limits_{i=1}^{\infty}f_i'(g)f_i$ converges unconditionally for each $g\in X$, we can define a bounded linear operator $\Psi_s:X \rightarrow \ell_s$ by
\begin{align*}
\Psi_s(g)=&\left(f_i'(g)\right)_{i\in\mathbb{N}} \;\;\;\;\text{for } g\in X.
\end{align*}
\item If $\left\lbrace f_i'\right\rbrace_{i\in\mathbb{N}}\subset X'$ is bounded below and $\sum\limits_{i=1}^{\infty}h'(f_i)f_i'$ converges unconditionally for each $h'\in X'$, then we can define a bounded linear operator $\Theta_{q}:X' \rightarrow \ell_q$  by
\begin{align*}
\Theta_q(h')=&\left(h'(f_i)\right)_{i\in\mathbb{N}} \;\;\;\;\text{for } h'\in X'.
\end{align*}
Let us write $\Phi_{q'}:= \Theta_{q}^{*}$. If we denote by
$\{e^{(i)}\}_{i\in\mathbb{N}}$ the unit vector basis of $\ell_{q'}$, it is easy to check that $\Phi_{q'}\left(e^{(j)}\right)=f_j$ for each $j\in \mathbb{N}$. Since $\Phi_{q'}$ is continuous, this implies that
for $\textbf{\textit{a}}=\left(a_i\right)_{i\in\mathbb{N}} \in \ell_{q'}$ we have
\begin{align*}
\Phi_{q'}(\textbf{\textit{a}})=&\Phi_{q'}\left(\sum\limits_{i=1}^{\infty}a_ie^{(i)}\right)=\sum\limits_{i=1}^{\infty}a_if_i . \end{align*}
\end{enumerate}
\end{remark}

A sequence $\{f_i\}_{i\in\mathbb{N}}$ in a Banach space $X$ is a \emph{Schauder basis} for $X$ if every element $g\in X$ has a unique expansion of the form
\begin{align}
g=&\sum\limits_{i=1}^{\infty}a_if_i, \label{bases}
\end{align}
with  $\{a_i\}_{i\in\mathbb{N}}\subset \mathbb K$. A sequence that is a Schauder basis for the closure of its span $\overline{\left[f_i:i\in\mathbb{N}\right]}$ is called a \textit{basic sequence}. A Schauder basis is called is \emph{unconditional} if the convergence in (\ref{bases}) is unconditional for every $g\in X$.

If $\{f_i\}_{i\in\mathbb{N}}$ is a Schauder basis for $X$, there is a uniquely determined sequence $\{f_i'\}_{i\in\mathbb{N}}$ in its dual space $X'$ such that $f_i'(f_j)=\delta_{ij}$ for every pair of positive integers $(i,j)$, and thus
\begin{align}
g=&\sum\limits_{i=1}^{\infty}f_i'(g)f_i \;\;\;\forall g\in X. \label{marco}
\end{align}
We call the sequence $\{f_i'\}_{i\in\mathbb{N}}$ the \emph{sequence of coordinate functionals} corresponding to the basis.

A sequence $\left\lbrace f_i,f_i'\right\rbrace_{i\in\mathbb{N}}\subseteq X\times X'$ for which (\ref{marco}) holds of is called a \emph{Schauder frame} for $X$. Note that, in this case, we do not require uniqueness of expansions like in \eqref{bases}. We call the sequence $\{f_i'\}_{i\in\mathbb{N}}$ the \emph{sequence of coordinate functionals} of the frame. If the convergence in (\ref{marco}) is unconditional for every $g\in X$, the Schauder frame is called \emph{unconditional}. Note that $f_i'(f_j)=\delta_{ij}$ for every $i,j \in \mathbb{N}$ if and only if $\{f_i\}_{i\in \mathbb{N}}$ is basic.

If $\left\lbrace f_i,f_i'\right\rbrace_{i\in\mathbb{N}}$ is a Schauder frame for $X$, there exists $K\ge 1$ such that for every $n\in \mathbb N$ and every $g\in X$ we have
\begin{align*}
\left|\left|\sum\limits_{i=1}^{n}f_i'(g)f_i\right|\right|_{X}\le& K||g||_X.
\end{align*}
The infimum of such $K$  is called the \emph{frame constant} of $\left\lbrace f_i,f_i'\right\rbrace_{i\in\mathbb{N}}$. If the frame is unconditional, there exists $K\ge 1$ such that
\begin{align*}
\left|\left|\sum\limits_{i=1}^{\infty}c_if_i'(g)f_i\right|\right|_{X}\le& K||g||_{X}||\emph{c}||_{\infty}.
\end{align*}
for all $ g\in X$ and $ \emph{\textbf{c}}=\left(c_i\right)_{i\in\mathbb{N}} \in \ell_{\infty}$.
The infimum of such constants $K$ is called the \emph{unconditional constant} of the frame $\left\lbrace f_i,f_i'\right\rbrace_{i\in\mathbb{N}}$.

In \cite{Freeman2014}, an extension of the concept of Schauder frames was introduced: a sequence $\left\lbrace f_i,f_i'\right\rbrace_{\in\mathbb{N}}\subseteq X\times X'$ is called an \emph{approximate Schauder frame} for $X$ if there is an isomorphism $S:X\rightarrow X$ given by
\begin{align}
S(f)=&\sum\limits_{i=1}^{\infty}f_i'(g)f_i \;\;\;\;\forall g\in X. \label{aproximado}
\end{align}
The operator $S$ is called the \emph{frame operator}. An approximate Schauder frame is said to be \emph{unconditional} if the convergence in (\ref{aproximado}) is unconditional for each $g \in X$. Note that, by the uniform boundedness principle, one can extend the concepts of frame constant and unconditional frame constant to approximate Schauder frames and unconditional approximate Schauder frames respectively, though unlike the case of frames, they might be smaller than one. We will define similar constants for a larger class of sequences: for each sequence $\left\lbrace f_i,f_i'\right\rbrace_{i\in\mathbb{N}}\subseteq X\times X'$ we define the \emph{frame constant} of $\left\lbrace f_i,f_i'\right\rbrace_{i\in\mathbb{N}}$ by
\begin{align}
K\left(\left\lbrace f_i,f_i'\right\rbrace_{i\in\mathbb{N}}\right)=&\max{\left\lbrace 1, \sup\limits_{n\in\mathbb{N}}\sup\limits_{g\in B_{X}}\left|\left|\sum\limits_{i=1}^nf_i'(g)f_i\right|\right|\right\rbrace},\nonumber
\end{align}
In a similar way, we define the \emph{unconditional frame constant} of $\left\lbrace f_i,f_i'\right\rbrace_{i\in\mathbb{N}}$ by
\begin{align}
K_u\left(\left\lbrace f_i,f_i'\right\rbrace_{i\in\mathbb{N}}\right)=&\max{\left\lbrace 1, \sup\limits_{n\in\mathbb{N}}\sup\limits_{g\in B_{X}}\sup\limits_{\textbf{c}\in B_{\ell_{\infty}}}\left|\left|\sum\limits_{i=1}^n c_if_i'(g)f_i\right|\right| \right\rbrace}.\nonumber
\end{align}
Note that, if $\sum\limits_{i=1}^{\infty}f_i'(g)f_i$ converges for every $g\in X$, then by the uniform boundedness principle, $K\left(\left\lbrace f_i,f_i'\right\rbrace_{i\in\mathbb{N}}\right)$ is finite. Similarly, if the series converges unconditionally for each $g\in X$, $K_u\left(\left\lbrace f_i,f_i'\right\rbrace_{i\in\mathbb{N}}\right)$ is finite.

\bigskip
The rest of this section is devoted to proving some technical results. Their main goal is to obtain, from a given approximate Schauder frame $\left\lbrace f_i,f_i'\right\rbrace_{i\in\mathbb{N}}$ for $X$, a new Schauder frame for $X$ with seminormalized coordinates, which preserves some properties of the original aproximate frame. We  need the following lemma from \cite{Freeman2014}.
\begin{lemma}\label{lemamarcoaproximado}\cite[Lemma 3.1]{Freeman2014}
Let $X$ be a Banach space. If $\left\lbrace f_i,f_i'\right\rbrace_{\in\mathbb{N}}\subseteq X\times X'$ is an approximate Schauder frame for $X$ with frame operator $S$, then $\left\lbrace f_i,\left(S^{-1}\right)^{*}f_i'\right\rbrace_{\in\mathbb{N}}$ is a Schauder frame for $X$. Moreover, if $\left\lbrace f_i,f_i'\right\rbrace_{\in\mathbb{N}}$ is unconditional, so is $\left\lbrace f_i,\left(S^{-1}\right)^{*}f_i'\right\rbrace_{\in\mathbb{N}}$.
\end{lemma}
In the proof of our next lemma, we first obtain an approximate Schauder frame with seminormalized coordinates from an approximate Schauder frame, by adding to the coordinate functionals a seminormalized sequence of functionals with sufficiently small norms, and then, by an application of Lemma \ref{lemamarcoaproximado}, we get a Schauder frame with seminormalized coordinate functionals from it.
\begin{lemma}\label{lemaseminormalizar}
Let $\left\lbrace f_{i}, f_i' \right\rbrace_{i\in \mathbb{N}}\subseteq X\times X'$ be an approximate Schauder frame for $X$ with $\{f_i\}_{i\in\mathbb{N}}$ bounded below. If there is a seminormalized sequence $\{g_i'\}_{i\in\mathbb{N}}\subseteq X'$ such that $\sum\limits_{i=1}^{\infty}g_i'(g)f_i$ converges unconditionally for each $g\in X$, then there is a seminormalized sequence $\{F_i'\}_{i\in\mathbb{N}}\subseteq X'$ such that $\left\lbrace f_{i}, F_i' \right\rbrace_{i\in \mathbb{N}}$ is a Schauder frame for $X$, which is unconditional if $\left\lbrace f_{i}, f_i' \right\rbrace_{i\in \mathbb{N}}$ is unconditional.
\end{lemma}
\begin{proof}
Let $S$ be the frame operator of $\left\lbrace f_{i}, f_i' \right\rbrace_{i\in \mathbb{N}}$. Since $S: X\rightarrow X$ is an isomorphism,  there exists $0<\delta_0<1$ such that every $T: X\rightarrow X$  $||S-T||\le \delta_0$ is also an isomorphism. Note that $\left|\left|f_i'(g)f_i\right|\right|\le 2K\left(\left\lbrace f_i,f_i'\right\rbrace_{i\in\mathbb{N}}\right)\|g\|$ and $\left|\left|g_i'(g)f_i\right|\right|\le K_u\left(\left\lbrace f_i,g_i'\right\rbrace_{i\in\mathbb{N}}\right) \|g\| $ for all $g\in X$ and all $i\in\mathbb{N}$, so both $\{f_i'\}_{i\in\mathbb{N}}$ and $\{f_i\}_{i\in\mathbb{N}}$ are bounded. Set
\begin{align}
K_1=\frac{\max{\left\lbrace ||S||, K_u\left(\left\lbrace f_{i}, g_i' \right\rbrace_{i\in \mathbb{N}}\right), \sup\limits_{j\in\mathbb{N}}{\left\lbrace \max{\left\lbrace ||g_j'||,\frac{1}{||g_j'||},||f_j||, ||f_j'||\right\rbrace}\right\rbrace}\right\rbrace}}{\delta_0}.\label{cotagrande}
\end{align}
For each $i\in\mathbb{N}$, let $G_i'=f_i'+b_ig_i'$, where
\begin{numcases}{b_i=}
\frac{1}{K_1}& if $\left|\left|f_i'\right|\right|< \frac{1}{2K_1^2}$;\nonumber\\
0 & otherwise. \nonumber
\end{numcases}
For each $g\in X$, both $\sum\limits_{\substack{i=1}}^{\infty}b_ig_i'(g)f_i$ and $\sum\limits_{i=1}^{\infty}f_i'(g)f_i$ are convergent. Hence, we can define a linear operator $T:X\rightarrow X$ by
\begin{align*}
T(g)=\sum\limits_{i=1}^{\infty}G_i'\left(g\right)f_i \;\;\;\forall g\in X.
\end{align*}
Since $|b_i|\le K_1^{-1}$ for each $i\in \mathbb{N}$,  we obtain
\begin{align*}
\left|\left|S(g)-T(g)\right|\right|&=
\left|\left|\sum\limits_{i=1}^{\infty}f_i'\left(g\right)f_i-
\sum\limits_{i=1}^{\infty}G_i'\left(g\right)f_i\right|\right|=
\left|\left|\sum\limits_{\substack{i=1}}^{\infty}b_ig_i'\left(g\right)f_i\right|\right|\\
&\le
{K_u\left(\left\lbrace f_{i}, g_i' \right\rbrace_{i\in \mathbb{N}}\right)} \left|\left|(b_i)_i\right|\right|_{\infty}\left|\left|g\right|\right| \le \delta_0 K_1 K_1^{-1} \|g\|=\delta_0\|g\|\, .
\end{align*}
Therefore, $T$ is an isomorphism, so $\left\lbrace f_{i}, G_i' \right\rbrace_{i\in \mathbb{N}}\subseteq X\times X'$ is an approximate Schauder frame for $X$. For each $i\in \mathbb{N}$, let $F_i'=\left(T^{-1}\right)^{*}\left(G_i'\right)$. By Lemma \ref{lemamarcoaproximado}, $\left\lbrace f_{i}, F_i' \right\rbrace_{i\in \mathbb{N}}$ is a Schauder frame for $X$. Note that $\left\lbrace G_i'\right\rbrace_{i\in\mathbb{N}}$ is bounded because both $\left\lbrace f_i'\right\rbrace_{i\in\mathbb{N}}$ and $\left\lbrace g_i'\right\rbrace_{i\in\mathbb{N}}$ are bounded. Also, from (\ref{cotagrande}) we get that $K_1^{-1}\le ||g_i'||$ for each $i\in \mathbb{N}$, so our choice of $\left\lbrace b_i\right\rbrace_{i\in\mathbb{N}}$ gives
\begin{align}
\left|\left|G_i'\right|\right|\ge& \frac{1}{2K_1^2}\;\;\;\;\forall i\in\mathbb{N}. \nonumber
\end{align}
Hence, $\left\lbrace G_i'\right\rbrace_{i\in\mathbb{N}}$ is seminormalized, and thus so is $\left\lbrace F_i'\right\rbrace_{i\in\mathbb{N}}$. If $\left\lbrace f_{i}, f_i' \right\rbrace_{i\in \mathbb{N}}$ is unconditional, then for every $g\in X$, the series $\sum\limits_{i=1}^{\infty}f_i'(g)f_i$ converges unconditionally. Since $\sum\limits_{i=1}^{\infty}b_ig_i'(g)f_i$ also converges unconditionally, it follows that $\left\lbrace f_{i}, G_i' \right\rbrace_{i\in \mathbb{N}}$ is unconditional and, by Lemma \ref{lemamarcoaproximado}, so is $\left\lbrace f_{i}, F_i' \right\rbrace_{i\in \mathbb{N}}$.
\end{proof}

\begin{corollary}\label{corolarioseminormalizar}
Let $\left\lbrace f_{i}, f_i' \right\rbrace_{i\in \mathbb{N}}\subseteq X\times X'$ be an approximate Schauder frame for $X$ such that $\left\lbrace f_i\right\rbrace_{i\in\mathbb{N}}$ is bounded below. Suppose that $X$ contains a complemented copy of $\ell_q$ ($1\le q<\infty$) and that
there is $M_0>0$ such that
\begin{align}
\left|\left|\sum\limits_{i=1}^{n}a_if_{i}\right|\right|\le&M_0\left(\sum\limits_{i=1}^{n}|a_i|^{q}\right)^{\frac{1}{q}}, \label{cotaeleq}
\end{align}
for each $n\in\mathbb{N}$ and every  $\left\lbrace a_i\right\rbrace_{1\le i\le n}\subset \mathbb K$.
Then, there is a seminormalized sequence $\{F_i'\}_{i\in\mathbb{N}}\subseteq X'$ such that $\left\lbrace f_{i}, F_i' \right\rbrace_{i\in \mathbb{N}}\subseteq X\times X'$ is a Schauder frame for $X$. Moreover, $\{F_i'\}_{i\in\mathbb{N}}$ can be chosen so that $\left\lbrace f_{i}, F_i' \right\rbrace_{i\in \mathbb{N}}$ is unconditional if $\left\lbrace f_{i}, f_i' \right\rbrace_{i\in \mathbb{N}}$ is unconditional.
\end{corollary}
\begin{proof}
Let $Y\subset X$ be a complemented copy of $\ell_q$ and $P:X\rightarrow Y$ be a bounded projection. Let $\{H_i\}_{i\in\mathbb{N}}$ be a basis of $Y$ equivalent to the unit vector basis of $\ell_q$, with coordinate functionals $\{H_i'\}_{i\in\mathbb{N}}\subseteq Y'$. By \eqref{cotaeleq}, there is a bounded operator $T:Y\to X$ such that $T(H_i)=f_i$. The basis $\{H_i\}_{i\in\mathbb{N}}$ is equivalent to the unit basis of $\ell_q$ and, then, the series $\sum_i H_i'(P(g)) H_i$ converges unconditionally for every $g\in X$. Since $T$ is bounded, the series $$\sum_i H_i'(P(g)) T(H_i)= \sum_i P^*(H_i'(g)) f_i$$ also converges unconditionally.
Since $P^*$ is bounded below, we can apply Lemma \ref{lemaseminormalizar} with  $g_i=P^*(H_i')$ to get the seminormalized sequence $\{F_i'\}_{i\in\mathbb{N}}\subseteq X'$ with the desired properties.
\end{proof}

\section{Unconditional Schauder frames of translates: definitions and general results.}\label{generalresults}
Fix a finite family $\left\lbrace g_k\right\rbrace_{1\le k\le k_0}\subset L_p(\mathbb{R}^d)$ (of different, nonzero functions), and a sequence $\left\lbrace \lambda_{i}\right\rbrace_{i\in\mathbb{N}}\subset \mathbb{R}^d$. We say that a sequence $\{f_i\}_{i\in\mathbb{N}}\subseteq L_p(\mathbb{R}^d)$ is a \emph{sequence of translates of $\left\lbrace g_k\right\rbrace_{1\le k\le k_0}$ by $\left\lbrace \lambda_{i}\right\rbrace_{i\in \mathbb{N}}\subseteq \mathbb{R}^d$} if, for each $i\in \mathbb N$, there exists $1\le k\le k_0$ such that $f_i$ is the translation of  $g_k$ by $\lambda_i$. In other words, there is a partition $\left\lbrace \Delta_k\right\rbrace_{1\le k\le k_0}$ of $\mathbb{N}$ into disjoint sets such that
\begin{align}
f_i=&T_{\lambda_i}g_k\;\;\;\forall i\in \Delta_k \;\;\forall 1\le k\le k_0. \label{condiciontraslacionesfinitas}
\end{align}
We may assume that each $\Delta_k$ is nonempty (if some $\Delta_k$ is empty, the corresponding $g_k$ is not used in the translations and we can remove it from the family).

\begin{definition}
  An approximate Schauder frame (or a Schauder frame) $\left\lbrace f_i,f_i'\right\rbrace_{i\in\mathbb{N}}$ for a subspace $X$ of $L_p\left(\mathbb{R}^d\right)$ is called an \emph{approximate Schauder frame of translates} (or a \emph{Schauder frame of translates})  if $\{f_i\}_{i\in\mathbb{N}}$ is a sequence of translates of finitely many functions by some sequence $\left\lbrace\lambda_{i}\right\rbrace_{i\in \mathbb{N}}\subseteq \mathbb{R}^d$.
\end{definition}

In this section, we study unconditional Schauder frames and approximate Schauder frames of translates of a single function or of finitely many functions in $L_p(\mathbb{R}^d)$, for $1<p<+\infty$. In particular, we focus on unconditional Schauder frames of translates with seminormalized coordinate functionals. We first introduce some definitions.
\begin{definition}\cite{Liu2012}\label{definicioncuasi}
An indexed family $\{\lambda_{i}\}_{i \in\Delta}\subseteq \mathbb{R}^d$ is \emph{uniformly separated} if there is $\delta>0$ such that
\begin{align}
 \left|\lambda_{i}-\lambda_{i'}\right|\ge \delta \;\;\;\forall i\not=i' \in \Delta, \nonumber
\end{align}
The set  is \emph{relatively uniformly separated} if there is a partition $\left\lbrace \Delta_k\right\rbrace_{1\le k\le m}$ of $\Delta$ such that $\{\lambda_{i}\}_{i \in\Delta_k}$ is uniformly separated for all $1\le k\le m$. \end{definition}

Note that all the elements of a uniformly separated family are different,  but this does not necessarily hold for relatively uniformly separated indexed sets.

\begin{remark}\label{lemmacaracterizarcuasiuniformementediscreto}
 If $\{\lambda_{i}\}_{i \in\Delta}\subseteq \mathbb{R}^d$ is relatively uniformly separated,  each compact set $Q\subset \mathbb{R}^d$  can contain only a finite number of $\lambda_i$'s. As a consequence, for any $t>0$ we can take a (further) partition
 $\left\lbrace A_k\right\rbrace_{1\le k\le k_0}$ of $\Delta$ such that
 \begin{align}
 \left|\lambda_i-\lambda_j\right| \ge  t \;\;\;\;\forall (i,j)\in A_k\times A_k: i\not=j\;\;\forall 1\le k\le k_0.\nonumber
\end{align}
\end{remark}

In   \cite[Proposition 1.8]{Odell2011}, the authors prove that if $\left\lbrace f_i,f_i'\right\rbrace_{i\in\mathbb{N}}$ is a biorthogonal system in $L_p\left(\mathbb{R}\right)^d$ with $\sup_{i\in \mathbb{N}}\|f_i\|\|f_i'\|<+\infty$ and $\left\lbrace f_i \right\rbrace_{i\in\mathbb{N}}$ is a sequence of translates of $f$ by a sequence $\left\lbrace \lambda_i\right\rbrace_{i\in\mathbb{N}}\subseteq \mathbb{R}^d$, then $\left\lbrace \lambda_i\right\rbrace_{i\in\mathbb{N}}$ is uniformly separated (this extends \cite[Theorem 1]{Olson1992}, where the corresponding result for Schauder bases is given). In \cite{Liu2012}, similar results are obtained for more general sequences. In particular, the following statement will be useful for our purposes.
\begin{theorem}\label{teoremadeLiu}\cite[Lemma 2.5 and Theorem 3.5 (i)]{Liu2012}
Let $q>1$, $K>0$ and $1<p<+\infty$. Suppose $\{f_i\}_{i\in\mathbb{N}}\subseteq X\subseteq L_p(\mathbb{R}^d)$ is a sequence of translates of $\{g_k\}_{1\le k\le k_0}$ by $\left\lbrace \lambda_j\right\rbrace_{j\in\mathbb{N}}$ and
\begin{align}
\sum\limits_{i=1}^{\infty}\left|h'\left(f_i\right)\right|^{q}\le K\left|\left|h'\right|\right|^{q}\;\;\;\;\forall h'\in X'.\nonumber
\end{align}
Then, $\left\lbrace \lambda_j\right\rbrace_{j\in\mathbb{N}}$ is relatively uniformly separated.
\end{theorem}

The following  lemma is probably known, but we could not find a reference for it.
\begin{lemma}\label{lemmadual}
Let $X$ be a reflexive Banach space, and let $\left\lbrace f_i,f_i'\right\rbrace_{i\in\mathbb{N}}$ be a sequence in $X\times X'$. The following are equivalent:
\begin{enumerate}[{(\roman*)}]
\item \label{desdeeldual}$\sum\limits_{i=1}^{\infty}h'(f_i)f_i'$ converges unconditionally for each $h'\in X'$.
\item \label{haciaeldual}$\sum\limits_{i=1}^{\infty}f_i'(g)f_i$ converges unconditionally for each $g\in X$.
\end{enumerate}
\end{lemma}
\begin{proof}
\ref{desdeeldual}$\Longrightarrow$\ref{haciaeldual}: Fix $g\in X$. For each $h'\in X'$, choose $\left(a_i(h')\right)_{i\in\mathbb{N}}\in B_{\ell_{\infty}}$ so that $a_i(h')h'(f_i)f_i'(g)=|h'(f_i)f_i'(g)|$. We have
\begin{align}
\sum\limits_{i=1}^{\infty}\left|h'\left(f_i'(g)f_i\right)\right|=
&\left|\sum\limits_{i=1}^{\infty}a_i(h')h'(f_i)f_i'(g)\right|=
\left|\left(\sum\limits_{i=1}^{\infty}a_i(h')h'(f_i)f_i'\right)(g)\right|<+\infty,
\end{align}
for all $h'\in X'$. Since $X$ does not contain a subspace isomorphic to $\textbf{c}_0$, by a theorem of Bessaga and Pe{\l}czy{\'n}ski the series $\sum\limits_{i=1}^{\infty}f_i'(g)f_i$ converges unconditionally  (see  \cite[Theorem 5]{Bessaga1958} or \cite[Theorem 6.4.3]{Kadets1997}). \\
\ref{haciaeldual} $\Longrightarrow$\ref{desdeeldual}: Since $X$ is reflexive, this is proven by essentially the same argument.
\end{proof}
\smallskip
If $\left\lbrace f_i,f_i'\right\rbrace_{i\in\mathbb{N}}$ is an unconditional approximate Schauder frame for a reflexive Banach space $X$, by the previous lemma $\sum\limits_{i=1}^{\infty}f'(f_i)f_i'$ converges unconditionally for each $f'\in X'$.
Thus, if $S:X\rightarrow X$ is the frame operator of $\left\lbrace f_i,f_i'\right\rbrace_{i\in\mathbb{N}}$, then for every $f'\in X'$ we have
$$
S^{*}f'(f)=f'\left(S\left(f\right)\right)=f'\left(\sum\limits_{i=1}^{\infty}f_i'(f)f_i\right)=\left(\sum\limits_{i=1}^{\infty}f'(f_i)f_i'\right)(f)\;\;\;\;\forall f\in X.
$$
Hence, $S^{*}f'=\sum\limits_{i=1}^{\infty}f'(f_i)f_i'$. Since $S^{*}:X'\rightarrow X'$ is an isomorphism, it follows that $\left\lbrace f_i',f_i\right\rbrace_{i\in\mathbb{N}}$ is an approximate unconditional Schauder frame for $X'$ with frame operator $S^{*}$. For unconditional Schauder frames, this result follows at once from \cite[Theorem 1.4, Theorem 2.5]{Carando2009} and \cite[Proposition 2.2, Remark 3.2]{Carando2011}. It was also proven in \cite[Theorem 5.1]{Liu2010}, under the assumption that for each $i\in\mathbb{N}$, $\left|\left|f_i'\big|_{\left[f_j:j\ge n\right]}\right|\right|\rightarrow 0$ as $n\to +\infty$.

\begin{corollary}\label{propositionmarcoincondicionalcud}
Let $1 <p< +\infty$. Suppose $\left\lbrace f_i, f_i' \right\rbrace_{i\in \mathbb{N}}\subseteq X\times X'$ is an unconditional approximate Schauder frame for $X\subseteq L_p(\mathbb{R}^d)$, and $\{f_i\}_{i\in \mathbb{N}}$ is a sequence of translates of $\{g_k\}_{1\le k\le k_0}$ by $\left\lbrace \lambda_i\right\rbrace_{i\in\mathbb{N}}$. If $\left\lbrace f_i'\right\rbrace_{i\in\mathbb{N}}$ is seminormalized, then $\left\lbrace \lambda_i\right\rbrace_{i\in \mathbb{N}}$ is relatively uniformly separated.
\end{corollary}
\begin{proof} By the previous comments and Remark
\ref{lemaacodadoabajovaaeleq2}, $h'\rightarrow \left(h'(f_i)\right)_{i\in\mathbb{N}}$ defines a bounded linear operator from $X'$ into $\ell_q$, where $q=\max\{2,p'\}$. Then, Theorem \ref{teoremadeLiu} shows that $\left\lbrace \lambda_i\right\rbrace_{i\in \mathbb{N}}$ is relatively uniformly separated.
\end{proof}

 As a consequence of our next lemma and its corollary,
for unconditional approximate Schauder frames of relatively uniformly separated  translates with seminormalized coordinates, we can strengthen some conclusions of Remark \ref{lemaacodadoabajovaaeleq}. The proof of the lemma is a variant of the proof of \cite[Lemma 2, Section~3]{Johnson1974}.
\begin{lemma}\label{lemaincondicionallp}
Let $1\le p<+\infty$, and let $\left\lbrace f_{i} \right\rbrace_{i\in \mathbb{N}}\subseteq L_p(\mathbb{R}^d)$. Suppose there is a partition $\left\lbrace A_k\right\rbrace_{1\le k\le k_0}$ of $\mathbb{N}$, $\epsilon>0$ and measurable sets $\left\lbrace D_{i,k}\right\rbrace_{\substack{1\le k\le k_0\\i\in A_k}}\subseteq \mathbb{R}^d$ such that
\begin{align}
D_{i,k}\cap D_{i',k}=&\emptyset \;\;\;\;\text{for } i,i'\in A_k,\  i\not=i'\;\;\text{ and }1\le k\le k_0;\nonumber\\
\int_{D_{i,k}}\left|f_i(x)\right|^{p}dx\ge& \epsilon \;\;\;\;\text{for } i\in A_k \;\;\text{ and } 1\le k\le k_0.\nonumber
\end{align}
If  $\sum\limits_{i=1}^{\infty}a_if_i$ converges unconditionally, then $\sum\limits_{i=1}^{\infty}\left|a_i\right|^p<\infty$.
\end{lemma}
\begin{proof}
Since $\sum\limits_{i=1}^{\infty}a_if_i$ converges unconditionally, there is $K>0$ such that
\begin{align*}
\left|\left|\sum\limits_{i=1}^{n}c_ia_if_i\right|\right|\le& K\left|\left|\textbf{\emph{c}}\right|\right| \;\;\;\text{for all } \textbf{\emph{c}}=\left(c_i\right)_{i\in\mathbb{N}} \in \ell_{\infty},\; n\in\mathbb{N}.
\end{align*}

As a consequence, for fixed $1\le k\le k_0$ and $n\in \mathbb N$ we have, using H\"older's inequality in the third step,
\begin{align*}
\epsilon\sum\limits_{\substack{j\in A_k\\1\le j\le n}}\left|a_j\right|^p & \le \sum\limits_{\substack{j\in A_k\\1\le j\le n}}\int_{D_{j,k}}\left|a_jf_j(x)\right|^{p}dx \\ &= \sum\limits_{j\in A_k}\int_{D_{j,k}}\Big|\int_{0}^{1}\sum\limits_{\substack{1\le i\le n\\i\in A_k}}a_if_i(x)r_i(t)r_j(t)dt\Big|^{p}dx \\
& \le \sum\limits_{j\in A_k}\int_{D_{j,k}} \int_{0}^{1}\Big|\sum\limits_{\substack{1\le i\le n\\i\in A_k}}a_if_i(x)r_i(t)\Big|^p|r_j(t)|^p dt\, dx \\
& = \sum\limits_{j\in A_k}\int_{D_{j,k}} \int_{0}^{1}\Big|\sum\limits_{\substack{1\le i\le n\\i\in A_k}}a_if_i(x)r_i(t)\Big|^p dt\, dx.
\end{align*}
Since $D_{j,k}\cap D_{j',k}$ is empty for $j\ne j'$, this last expression is bounded by
\color{black}

\begin{align*}
\int_{\mathbb{R}^d}\int_{0}^{1}\Big|\sum\limits_{\substack{1\le i\le n\\i\in A_k}}r_i(t)a_if_i(x)\Big|^p dt\, dx = \int_{0}^{1}\Big\|\sum\limits_{\substack{1\le i\le n\\i\in A_k}}r_i(t)a_if_i\Big\|^{p}dt\le K^p,
\end{align*}
Therefore,
$$
\sum\limits_{i=1}^{\infty}\left|a_i\right|^p=\sum\limits_{k=1}^{k_0}\sum\limits_{\substack{i\in A_k}}\left|a_i\right|^p\le \frac{k_0K^p}{\epsilon}<+\infty. \qedhere
$$

\end{proof}

\begin{corollary}\label{corolariooperadoralp}
Let $1\le p<+\infty$, and let $\left\lbrace f_{i}, f_i' \right\rbrace_{i\in \mathbb{N}}$ be an unconditional approximate Schauder frame of translates for $X\subseteq L_p(\mathbb{R}^d)$ by a relatively uniformly separated sequence.
Then there exists a bounded linear operator $\Psi_p:X\rightarrow \ell_p$ given by
\begin{align}
\Psi_p(g)=&\left(f_i'(g)\right)_{i\in\mathbb{N}} \;\;\;\;\forall g\in X.\label{continuoaelepe}
\end{align}
\end{corollary}
\begin{proof} Suppose $\{f_i\}_{i\in \mathbb{N}}$ is a sequence of translates of $\{g_k\}_{1\le k\le k_0}$ by $\left\lbrace \lambda_i\right\rbrace_{i\in\mathbb{N}}$.
Since all $g_k\ne 0$ for  $1\le k\le k_0$, there exists $\epsilon>0$ and a cube $Q\subseteq \mathbb{R}^d$ such that
\begin{align*}
\int_{Q}\left|g_k(x)\right|^pdx\ge&\epsilon\;\;\;\;\forall 1\le k\le k_0.
\end{align*}
By Remark \ref{lemmacaracterizarcuasiuniformementediscreto} there is a partition $\{A_m\}_{1\le m\le m_0}$ of $\mathbb{N}$ such that
\begin{align}
Q+\lambda_i\cap Q+\lambda_j=&\emptyset \;\;\;\forall (i,j)\in A_m\times A_m: i\not=j \;\forall 1\le m\le m_0. \nonumber
\end{align}
Let $\{\Delta_k\}_{1\le k\le k_0}$ be a partition of $\mathbb{N}$ such that
\begin{align*}
f_i=&T_{\lambda_i}g_k\;\;\;\forall i\in \Delta_k \;\;\forall 1\le k\le k_0,
\end{align*}
and let
\begin{align*}
D_{i,m,k}=&Q+\lambda_{i} \;\;\;\;\forall i\in A_m\cap \Delta_k \;\;\forall 1\le m\le m_0\;\;\forall 1\le k\le k_0.
\end{align*}
Fix $1\le k\le k_0$ and $1\le m\le m_0$. We have
\begin{align*}
D_{i,m,k}\cap D_{i',m,k}=&\emptyset \;\;\;\;\forall (i,i')\in A_m\cap \Delta_k\times A_m\cap \Delta_k :i\not=i';\\
\int_{D_{i,m,k}}\left|f_i(x)\right|^{p}dx=&\int_{Q+\lambda_i}\left|T_{\lambda_{i}}g_k(x)\right|^{p}dx=\int_{Q}\left|g_k(x)\right|^{p}dx\ge \epsilon \;\;\;\forall i\in A_m\cap \Delta_k.
\end{align*}
By Lemma \ref{lemaincondicionallp}, this implies that $\sum\limits_{i=1}^{\infty}\left|f_i'(g)\right|^p<+\infty$ for each $g\in X$. A uniform boundedness argument completes the proof.
\end{proof}
Note that for the case $2\le p<+\infty$, the result of Lemma \ref{lemaincondicionallp} follows immediately from Proposition \ref{TeoremadeOrliczextendido}, but it is convenient for our purposes to state the lemma and its corollary for all $1\le p <+\infty$.

\smallskip
Now we prove the main result of this section.
\begin{proposition}\label{proposicionisomorfoaelep}
Let $\left\lbrace f_{i}, f_i' \right\rbrace_{i\in \mathbb{N}}$ be an unconditional approximate Schauder frame of translates for $X\subseteq L_p(\mathbb{R}^d)$, with $1<p<+\infty$, by $\left\lbrace \lambda_i\right\rbrace_{i\in\mathbb{N}}$. If there exists $M_0>0$ such that
\begin{align*}
\left|\left|\sum\limits_{i=1}^{n}a_if_{i}\right|\right|\le&M_0\left(\sum\limits_{i=1}^{n}|a_i|^{p}\right)^{\frac{1}{p}}
\end{align*}
for each $n\in\mathbb{N}$ and every $\{a_i\}_{1\le i\le n}\subset \mathbb K$, then, the following assertions hold.
\begin{enumerate}[label=\rm{(\alph*)}]
\item The sequence $\left\lbrace \lambda_i\right\rbrace_{i\in\mathbb{N}}$ is relatively uniformly separated.
\item The linear operator $g\rightarrow \left(f_i'(g)\right)_{i\in\mathbb{N}}$ is an isomorphism between $X$ and a complemented subspace of $\ell_p$. In particular, $X$ is isomorphic to $\ell_p$.
\item \label{seminormalizados29}There is a seminormalized sequence $\left\lbrace F_i'\right\rbrace_{i\in\mathbb{N}}$ such that $\left\lbrace f_{i}, F_i' \right\rbrace_{i\in \mathbb{N}}$ is an unconditional Schauder frame for $X$.

\end{enumerate}
\end{proposition}
\begin{proof} Suppose that $\{f_i\}_{i\in \mathbb{N}}$ is a sequence of translates of $\{g_k\}_{1\le k\le k_0}$ by $\left\lbrace \lambda_i\right\rbrace_{i\in\mathbb{N}}$
Let $\Phi_p:\ell_p\rightarrow X$ be the bounded linear operator given by
\begin{align}
\Phi_{p}(\textbf{\emph{a}})=\sum\limits_{i=1}^{\infty}a_if_i \;\;\;\;\text{for } \textbf{\emph{a}} \in \ell_{p}.\label{operadordeelep}
\end{align}
If $\left\lbrace e^{(i)}\right\rbrace_{i\in\mathbb{N}}$ is the unit vector basis of $\ell_{p}$, for each $h'$ in $X'$ we have
\begin{align*}
\sum\limits_{i=1}^{+\infty}\left|h'\left(f_i\right)\right|^{p'}=&\sum\limits_{i=1}^{+\infty}\left|\Phi_{p}^{*}\left(h'\right)\left(e^{(i)}\right)\right|^{p'}=\left|\left|\Phi_p^{*}\left(h'\right)\right|\right|_{\ell_{p'}}^{p'}\le \left|\left|\Phi_p\right|\right|^{p'}\left|\left|h'\right|\right|^{p'}\ \le M_0^{p'}||h'||^{p'}\,.
\end{align*}
Hence, by Theorem \ref{teoremadeLiu}, $\left\lbrace \lambda_{i}\right\rbrace_{i\in\mathbb{N}}$ is relatively uniformly separated and, by Corollary \ref{corolariooperadoralp}, we can define a bounded linear operator $\Psi_p:X\rightarrow \ell_p$ by
\begin{align}
\Psi_p(g)=&\left(f_i'(g)\right)_{i\in\mathbb{N}} \;\;\;\;\text{for } g\in X.\nonumber
\end{align}
Let $S$ be the frame operator of $\left\lbrace f_{i}, f_i' \right\rbrace_{i\in \mathbb{N}}$, and take $Z=\Psi_p(X)$. For each $g\in X$, we have
\begin{align}\label{lainversadeelep}
S^{-1}\circ \Phi_p\big|_{Z} \left(\Psi_p(g)\right)=&S^{-1}\left(\Phi_{p}\left(\left(f_i'(g)\right)_{i\in\mathbb{N}}\right)\right)=S^{-1}\left(\sum\limits_{i=1}^{\infty}f_i'(g)f_i\right)\\ & =S^{-1}\left(S(g)\right)=g. \nonumber
\end{align}
This shows that $\Psi_p$ is an isomorphism between $X$ and $Z$. If $\textbf{\emph{a}}=\left(a_i\right)_{i\in\mathbb{N}}\in Z$, there is $g\in X$ such that $\Psi_p(g)=\textbf{\emph{a}}$. Thus, it follows from (\ref{operadordeelep}) and (\ref{lainversadeelep}) that
\begin{align}
\textbf{\emph{a}}=&\Psi_p(g)=\Psi_p\left(S^{-1}\circ \Phi_p \left(\Psi_p(g)\right)\right)=\Psi_p\left(S^{-1}\circ \Phi_p \left(\textbf{\emph{a}}\right)\right)=\left(\Psi_p\circ S^{-1}\circ \Phi_p\right)\left(\textbf{\emph{a}}\right). \nonumber
\end{align}
Therefore, $\Psi_p\circ S^{-1}\circ \Phi_p: \ell_p\rightarrow Z$ is a bounded projection. Since every complemented subspace of $\ell_p$ is isomorphic to $\ell_p$ (\cite[Theorem 2.a.3]{Lindenstrauss1977}), $X$ is isomorphic to $\ell_p$.

Now,  \ref{seminormalizados29} is a direct  application of Corollary \ref{corolarioseminormalizar}.
\end{proof}

\section{Schauder frames of translates in $L_p(\mathbb R^d)$.}\label{results_in_L_p}

\subsection{The case $2<p<+\infty$.}\label{2<p}

The main result of this section is the following strengthened version of Theorem \ref{teoremamarcoparap>2} that gives  unconditional frames of translates for $L_p\left(\mathbb{R}^d\right)$ with seminormalized coordinate functionals.

\begin{theorem}\label{Teoremamarcosincondicionalescon2<p<+infinito}
Let $2<p<+\infty$ and let $\left\lbrace \lambda_i \right\rbrace_{i\in \mathbb{N}}\subseteq \mathbb{R}^d$ be an unbounded sequence. There is a subsequence $\left\lbrace \lambda_{m_i}\right\rbrace_{i\in \mathbb{N}}$, a function $f\in L_p(\mathbb{R}^d)$, and a seminormalized sequence $\{F_i'\}_{i\in\mathbb{N}}\subseteq L_{p'}(\mathbb{R}^d)$ such that $\left\lbrace T_{\lambda_{m_i}}f, F_i' \right\rbrace_{i\in \mathbb{N}}$ is an unconditional Schauder frame for $L_p(\mathbb{R}^d)$.
\end{theorem}

\begin{proof}
We first present a sketch of the main steps in the proof of Theorem \ref{teoremamarcoparap>2} from \cite[Theorem 3.2]{Freeman2014}.
\begin{enumerate}[\textbf{(\alph*)}]
\item \label{baseincondicional} Choose a normalized unconditional Schauder basis $\left\lbrace h_i\right\rbrace$ for $L_p(\mathbb{R}^d)$ such that \allowbreak $\diam{\left(\supp{\left(h_i\right)}\right)}\le 1$ for each $i\in \mathbb{N}$, and a sequence $\left\lbrace N_k\right\rbrace_{k\in\mathbb{N}}\subseteq \mathbb{N}$ such that
\begin{align}
\left(\sum\limits_{k=1}^{\infty}N_k^{1-\frac{p}{2}}\right)^{\frac{1}{p}}<\frac{1}{2K_u\left(\left\lbrace h_i,h_i'\right\rbrace_{i\in\mathbb{N}}\right)},\label{cotadelasuma}
\end{align}
where $\left\lbrace h_i'\right\rbrace_{i\in\mathbb{N}}$ is the sequence of coordinate functionals associated with the basis $\left\lbrace h_i\right\rbrace_{i\in\mathbb{N}}$.
\item \label{subsucesionquesepara} Choose a sequence of positive integers $\left\lbrace j_s^{(k)}\right\rbrace_{\substack{k\in\mathbb{N}\\1\le s\le N_k}}$ with the following properties:
\begin{enumerate}[\textbf{$($\roman*$)$}]
\item $j_s^{(k)}<j_{s'}^{(k')}\;\forall k<k'\;\forall 1\le s\le N_k'\;\forall 1\le s'\le k'$.
\item $j_s^{(k)}<j_{s'}^{(k)}\;\forall k\in\mathbb{N}\;\forall 1\le s<s'\le N_k$.
\item $\left|\lambda_{j_1^{(1)}}\right|>1$, and
\begin{align*}
\left|\lambda_{j_s^{(k)}}\right|>&3\max\left\lbrace \left|\lambda_{j_{s'}^{(k')}}\right|:j_{s'}^{(k')}<j_s^{(k)}\right\rbrace+2\max\left\lbrace \left|x\right|:x\in \overline{\supp{\left(h_j\right)}}, 1\le j\le k\right\rbrace  \\
&\forall k\in\mathbb{N} \; \forall 1\le s\le N_k.
\end{align*}
\end{enumerate}
\end{enumerate}
From these choices, the sets
\begin{align*}
J_k=\left\lbrace j_s^{(k)}: 1\le s\le N_k\right\rbrace
\end{align*}
are pairwise disjoint.

Let $\mathfrak{D}=\bigcup\limits_{k=1}^{\infty}J_k$, and
\begin{numcases}{k_i=}
k & if $i\in J_k$; \nonumber\\
0 & if $i\not\in \mathfrak{D}$. \nonumber
\end{numcases}
It follows that
\begin{align}
\supp{\left(T_{\lambda_i-\lambda_j} h_{k_j}\right)}\cap\supp{\left(T_{\lambda_{i'}-\lambda_{j'}}h_{k_{j'}}\right)}=&\emptyset \label{interseccionvacia}
\end{align}
for every $j,j',i,i'\in \mathfrak{D}$ with $i\not=j$, $i'\not=j'$, and $(i,j)\not=(i',j')$.
Define
\begin{align*}
f=\sum\limits_{k=1}^{\infty}\sum\limits_{j\in J_k}N_k^{-\frac{1}{2}}T_{-\lambda_{j}}h_k,
\end{align*}
and for each $i\in \mathbb{N}$,
\begin{numcases}{f_i'=}
N_k^{-\frac{1}{2}}h_k' & if $i\in J_k$\nonumber\\
0 & if $i\not\in \mathfrak{D}$. \nonumber
\end{numcases}
With this, $\left\lbrace T_{\lambda_i}f,f_i'\right\rbrace$ is an unconditional approximate Schauder frame for $L_p(\mathbb{R}^d)$.
\medskip
Now we modify this proof to obtain our result. First, we write $$\mathfrak{D}=\{m_i\}_{i\in \mathbb N},$$ with  $\{m_i\}_i$ an increasing sequence of natural numbers. Since $\left\lbrace T_{\lambda_{i}}f,f_i'\right\rbrace_{i\in\mathbb{N}}$ is an unconditional approximate Schauder frame for $L_p(\mathbb{R}^d)$ and $f_i'=0$ for all $i\not\in \mathfrak{D}$, then clearly $\left\lbrace T_{\lambda_{m_i}}f,f_{m_i}'\right\rbrace_{i\in\mathbb{N}}$ is an unconditional approximate Schauder frame for $L_p(\mathbb{R}^d)$.
By Remark \ref{lemaacodadoabajovaaeleq2}, we can define a bounded linear operator $\Phi_2:\ell_2\rightarrow L_p(\mathbb{R}^d)$  by
\begin{align}
\Phi_{2}(\textbf{\emph{a}})=\sum\limits_{k=1}^{\infty}a_kh_k \;\;\;\;\forall \textbf{\emph{a}}=\left(a_i\right)_{i\in\mathbb{N}} \in \ell_{2}.\label{continuodesdeele2}
\end{align}
For $A$ a finite subset of $\mathfrak{D}$ and $\{b_i\}_{i\in\mathbb{N}}\subset\mathbb K$ we have
\begin{align}
\left|\left|\sum\limits_{i\in A}b_i T_{\lambda_{i}}f\right|\right|=&\left|\left|\sum\limits_{i\in A}b_i\sum\limits_{k=1}^{\infty}\sum\limits_{j\in J_{k}}N_k^{-\frac{1}{2}}T_{\lambda_i-\lambda_j}h_k\right|\right|\nonumber\\
\le&\left|\left|\sum\limits_{i\in A}b_i\sum\limits_{\substack{k=1}}^{\infty}\sum\limits_{\substack{j\in J_k\\j\not=i}}N_k^{-\frac{1}{2}}T_{\lambda_i-\lambda_j}h_k\right|\right|+\left|\left|\sum\limits_{i\in A}b_iN_{k_i}^{-\frac{1}{2}}h_{k_i}\right|\right|\,.\label{separaenAyB}
\end{align}
Next, we estimate each of the terms on the right-hand side of (\ref{separaenAyB}). From  (\ref{cotadelasuma}), (\ref{interseccionvacia}) and the fact that $\{h_i\}_{i\in\mathbb{N}}$ is normalized, we have
\begin{align}
&\left|\left|\sum\limits_{i\in A}b_i\sum\limits_{\substack{k=1}}^{\infty}\sum\limits_{\substack{j\in J_k\\j\not=i}}N_k^{-\frac{1}{2}}T_{\lambda_i-\lambda_j}h_k\right|\right|=\nonumber\\
=&\left(\sum\limits_{i\in A}\sum\limits_{\substack{k=1}}^{\infty}\sum\limits_{\substack{j\in J_{k}\\j\not=i}}\int_{\supp{\left(T_{\lambda_{i}-\lambda_{j}}h_{k}\right)}}\left|b_{i}N_{k}^{-\frac{1}{2}}\left(T_{\lambda_{i}-\lambda_{j}}h_{k}\right)(x)\right|^{p}dx\right)^{\frac{1}{p}}\nonumber\\
=&\left(\sum\limits_{i\in A}\sum\limits_{\substack{k=1}}^{\infty}\sum\limits_{\substack{j\in J_k\\j\not=i}}\left|b_i\right|^{p}N_k^{-\frac{p}{2}}\right)^{\frac{1}{p}}\le\left(\sum\limits_{i\in A}\left|b_i\right|^{p}\sum\limits_{\substack{k=1}}^{\infty}N_k^{1-\frac{p}{2}}\right)^{\frac{1}{p}}\nonumber\\
\le&\left(\sum\limits_{i\in A}\left|b_i\right|^{2}\right)^{\frac{1}{2}}.\label{primeracota2}
\end{align}
From (\ref{continuodesdeele2}), we also  have
\begin{align}
\left|\left|\sum\limits_{i\in A}b_iN_{k_i}^{-\frac{1}{2}}h_{k_i}\right|\right|=&\left|\left|\sum\limits_{\substack{k\in\mathbb{N}\\A\cap J_k\not=\emptyset}}\left(\sum\limits_{i\in A\cap J_k}b_iN_{k}^{-\frac{1}{2}}\right)h_k\right|\right| \\ \le & \left|\left|\Phi_2\right|\right|\left(\sum\limits_{\substack{k\in\mathbb{N}\\A\cap J_k\not=\emptyset}}\left|\sum\limits_{i\in A\cap J_k}b_iN_{k}^{-\frac{1}{2}}\right|^{2}\right)^{\frac{1}{2}}\nonumber\\
\le&\left|\left|\Phi_2\right|\right|\left(\sum\limits_{\substack{k\in\mathbb{N}\\A\cap J_k\not=\emptyset}}\left(\sum\limits_{i\in A\cap J_k}\left|b_i\right|^{2}\right)\left(\sum\limits_{i\in A\cap J_k}N_k^{-1}\right)\right)^{\frac{1}{2}}\nonumber\\
\le& \left|\left|\Phi_2\right|\right|\left(\sum\limits_{\substack{k\in\mathbb{N}\\A\cap J_k\not=\emptyset}}\left(\sum\limits_{i\in A\cap J_k}\left|b_i\right|^{2}\right)\right)^{\frac{1}{2}}\nonumber\\
=& \left|\left|\Phi_2\right|\right|\left(\sum\limits_{i\in A}\left|b_i\right|^{2}\right)^{\frac{1}{2}}.\label{terceracota2}
\end{align}
From (\ref{separaenAyB}), (\ref{primeracota2}), and (\ref{terceracota2}) we get
\begin{align*}
\left|\left|\sum\limits_{i\in A}b_i T_{\lambda_{i}}f\right|\right|\le&\left(1+\left|\left|\Phi_2\right|\right|\right)\left(\sum\limits_{i\in A}\left|b_i\right|^{2}\right)^{\frac{1}{2}}.
\end{align*}
Thus, for every $n\in\mathbb{N}$ and for every set of scalars $\{a_i\}_{1\le i\le n}$, we have
\begin{align*}
\left|\left|\sum\limits_{i=1}^{n}a_i T_{\lambda_{m_i}}f\right|\right|\le&\left(1+\left|\left|\Phi_2\right|\right|\right)\left(\sum\limits_{i=1}^{n}\left|a_i\right|^{2}\right)^{\frac{1}{2}}.
\end{align*}
To finish the proof, let $M_0=1+\left|\left|\Phi_2\right|\right|$ and apply Corollary \ref{corolarioseminormalizar}
to the unconditional approximate Schauder frame $\left\lbrace T_{\lambda_{m_i}}f, f_{m_i}'\right\rbrace_{i\in\mathbb{N}}$.
\end{proof}

A question that arises naturally in this context is whether the seminormalized sequence of Theorem \ref{Teoremamarcosincondicionalescon2<p<+infinito} can be chosen so that $\left|F_i'(f_i)\right|\ge r$ for some $r>0$. The answer is negative, as the following extension of Theorem \ref{teoremabases2<p<+infinito} shows.
\begin{proposition}\label{proposicioncomobases2<p<+infinito}
Let $2<p<+\infty$, and let $\left\lbrace f_i,f_i' \right\rbrace_{i\in\mathbb{N}}\subseteq X\times X'$ be an unconditional approximate Schauder frame of translates for $X\subseteq L_p(\mathbb{R}^d)$.

If $X$ is complemented in $L_p(\mathbb{R}^d)$ and there is $r>0$ such that $r\le |f_i'(f_i)|$ for each $i \in \mathbb{N}$, then $g\rightarrow \left(f_i'(g)\right)_{i\in\mathbb{N}}$ defines an isomorphism between $X$ and a complemented subspace of $\ell_p$. Thus, $X$ is isomorphic to $\ell_p$.

\end{proposition}
\begin{proof}
Since $\left\lbrace f_i\right\rbrace_{i\in\mathbb{N}}$ is seminormalized, $\left\lbrace f_i'\right\rbrace_{i\in\mathbb{N}}$ is bounded.
Let $P:L_p(\mathbb{R}^d)\rightarrow X$ be a bounded projection. Suppose that $\{f_i\}_{i\in \mathbb{N}}$ is a sequence of translates of $\{g_k\}_{1\le k\le k_0}$ by $\left\lbrace \lambda_i\right\rbrace_{i\in\mathbb{N}}$ and choose $M>0$ and a cube $Q\subseteq \mathbb{R}^d$ such that
\begin{align}
\max\{||P||||f_i'||, ||f_i||\} \le M\;\;\;\forall i\in\mathbb{N}; \label{cotaparatodos}\\
\left|\left|g_k\big|_{\mathbb{R}^d\setminus Q}\right|\right|\le \frac{r}{2M} \;\;\;\forall 1\le k\le k_0. \label{cotagkafueradelcubo}
\end{align}
Given that $r\le |f_i'(f_i)|$ for each $i \in \mathbb{N}$, $\left\lbrace f_i'\right\rbrace_{i\in\mathbb{N}}$ is bounded below, and thus seminormalized. Hence, by Corollary~\ref{propositionmarcoincondicionalcud}, $\left\lbrace \lambda_i\right\rbrace_{i\in\mathbb{N}}$ is relatively uniformly separated. By Remark \ref{lemmacaracterizarcuasiuniformementediscreto}  there is a partition $\left\lbrace A_m\right\rbrace_{1\le m\le m_0}$ of $\mathbb{N}$ such that
\begin{align}
Q+\lambda_{i}\cap Q+\lambda_{j}=&\emptyset \;\;\;\;\forall (i,j)\in A_m\times A_m: i\not=j\;\;\forall 1\le m\le m_0.\label{separados2}
\end{align}
Let $\left\lbrace \Delta_k \right\rbrace_{1\le k\le k_0}$ be a partition of $\mathbb{N}$ for which (\ref{condiciontraslacionesfinitas}) holds, and fix $1\le k\le k_0$ and $1\le m\le m_0$. It follows from (\ref{cotaparatodos}) and (\ref{cotagkafueradelcubo}) that
\begin{align*}
r\le&\left|f_i'\left(f_i\right)\right|=\left|f_i'\left(P(f_i)\right)\right|=\left|P^{*}f_i'\left(f_i\right)\right|=\left|\int_{\mathbb{R}^d}\left(P^{*}f_i'\right)(x)f_i(x)dx\right|\\
\le& \int_{\left(\mathbb{R}^d\setminus Q\right)+\lambda_i}\left|\left(P^{*}f_i'\right)(x)g_k\left(x-\lambda_i\right)\right|dx+\left|\left|P^{*}f_i'\big|_{Q+\lambda_i}\right|\right|\left|\left|f_i\big|_{Q+\lambda_i}\right|\right|\\
\le&\left|\left|P^{*}f_i'\right|\right|\left|\left|g_k\big|_{\mathbb{R}^d\setminus Q}\right|\right|+M\left|\left|P^{*}f_i'\big|_{Q+\lambda_i}\right|\right|\le \frac{r}{2}+M\left|\left|P^{*}f_i'\big|_{Q+\lambda_i}\right|\right|,
\end{align*}
for all $i\in \Delta_{k}\cap A_m$. Thus,
\begin{align}
\frac{r}{2M}\le&\left|\left|P^{*}f_i'\big|_{Q+\lambda_i}\right|\right|\;\;\;\;\forall i\in \Delta_{k}\cap A_m.\label{cotainferiorcomobases2<p}
\end{align}
Since the series $\sum\limits_{i=1}^{\infty}h'(f_i)f_i'$ converges unconditionally for each $h'\in X'$ and $P^{*}$ is bounded, so does $\sum\limits_{i=1}^{\infty}h'(f_i)P^{*}f_i'$. From this fact, (\ref{separados2}), (\ref{cotainferiorcomobases2<p}) and Lemma \ref{lemaincondicionallp}, we deduce that there is linear operator $\Theta_{p'}:X'\rightarrow \ell_{p'}$ given by
\begin{align*}
\Theta_{p'}(h')=&\left(h'\left(f_i\right)\right)_{i\in\mathbb{N}},
\end{align*}
which is bounded by the uniform boundedness principle. As in Remark \ref{lemaacodadoabajovaaeleq2}, this implies that there is a bounded linear operator $\Phi_p:\ell_p\rightarrow X$ given by
\begin{align}
\Phi_{p}(\textbf{\emph{a}})=\sum\limits_{i=1}^{\infty}a_if_i \;\;\;\;\forall \textbf{\emph{a}}=\left(a_i\right)_{i\in\mathbb{N}} \in \ell_{p}.\label{operadordeelep29}
\end{align}
Now we apply Proposition \ref{proposicionisomorfoaelep} to complete the proof.
\end{proof}
If, in addition to the hypotheses of Proposition \ref{proposicioncomobases2<p<+infinito}, we add the condition that $\left\lbrace f_i\right\rbrace_{i\in\mathbb{N}}$ be basic, it follows from (\ref{operadordeelep29}) and the uniqueness of the coefficients that $\Phi_{p}$ is injective. Since it is also surjective and bounded, it is an isomorphism between $\ell_p$ and $X$. Thus, $\left\lbrace f_i\right\rbrace_{i\in\mathbb{N}}$ is equivalent to the unit vector basis of $\ell_p$. The condition that $X$ be complemented in $L_p\left(\mathbb{R}^d\right)$ is essential in this context. In fact, it was proven in \cite[Example 2.16]{Odell2011} that for $p>2$, there are unconditional basic sequences of translates that are not equivalent to the unit vector basis of $\ell_p$.

\subsection{The case $1<p\le 2$.}\label{1<p<=2}
Unlike the case $2<p<+\infty$, for $1<p\le 2$ Problem \ref{pregunta1} has a negative answer for every sequence of translates. For $1<p<2$, this is a consequence of a more general result, which we prove next.

\begin{theorem}\label{Teorema 1<p<2}
Let $1<p\le 2$, and let $X$ be a subspace of $L_p(\mathbb{R}^d)$. Suppose $\left\lbrace f_{i}, f_i' \right\rbrace_{i\in \mathbb{N}}$ is an unconditional approximate Schauder frame of translates for $X$.

The following are equivalent.
\begin{enumerate}[label=\rm{(\roman*)}]
\item \label{hayseminormalizada1<p<2} There is a seminormalized sequence $\{F_i'\}_{i\in\mathbb{N}}\subseteq X'$ such that $\{f_i,F_i'\}_{i\in\mathbb{N}}$ is an unconditional Schauder frame for $X$.
\item \label{cota1<p<2} There is $M_0>0$ such that
\begin{align}
\left|\left|\sum\limits_{i=1}^{n}a_if_{i}\right|\right|\le&M_0\left(\sum\limits_{i=1}^{n}|a_i|^{p}\right)^{\frac{1}{p}} \label{cotaele22}
\end{align}
for each $n\in\mathbb{N}$ and every finite set of scalars $\{a_i\}_{1\le i\le n}$.
\end{enumerate}
Under these conditions, $g\rightarrow \left(f_i'(g)\right)_{i\in\mathbb{N}}$ defines an isomorphism between $X$ and a complemented subspace of $\ell_p$. Thus, $X$ is isomorphic to $\ell_p$.
\end{theorem}
\begin{proof}To see that
\ref{hayseminormalizada1<p<2}$\Longrightarrow$\ref{cota1<p<2}, we apply Remark \ref{lemaacodadoabajovaaeleq2} to $\left\lbrace f_i,F_i'\right\rbrace_{i\in\mathbb{N}}$ and set $M_0=\left|\left|\Phi_{p}\right|\right|$.
The implication
\ref{cota1<p<2}$\Longrightarrow$\ref{hayseminormalizada1<p<2} follows immediately by Proposition \ref{proposicionisomorfoaelep}.
Also by Proposition \ref{proposicionisomorfoaelep}, we conclude that $g\rightarrow \left(f_i'(g)\right)_{i\in\mathbb{N}}$ is an isomorphism between $X$ and a complemented subspace of $\ell_p$.
\end{proof}

If, in addition to the hypotheses of Theorem \ref{Teorema 1<p<2} we assume that $\left\lbrace f_i\right\rbrace_{i\in\mathbb{N}}$ is a basic sequence, then we see that $\Phi_{p}$ is an isomorphism between $\ell_p$ and $X$, so $\left\lbrace f_i\right\rbrace_{i\in\mathbb{N}}$ is equivalent to the unit vector basis of $\ell_p$. This gives an extension of Theorem \ref{teoremabaseequivalente12} for $L_p\left(\mathbb{R}^d\right)$ with $d\ge 1$ and for sequences of translates of finitely many functions.

Note that Theorem \ref{Teorema 1<p<2} does not give an answer to Problem \ref{pregunta1} for $p=2$. A negative answer will follow from more general results about the compactness of restriction operators, which are variants of similar results from \cite{Freeman2014} and \cite{Liu2012}.
We first state a result that follows immediately from \cite[Lemma 3.4 (i)]{Liu2012}. A similar result result was also proven in \cite{Odell2011}.
\begin{lemma}\label{lematrasladados3}
Let $1 \le p< +\infty$, and let $\left\lbrace \lambda_{i}\right\rbrace_{i\in \mathbb{N}} \subseteq \mathbb{R}^d$ be a relatively uniformly separated sequence. If $\{f_i\}_{i\in\mathbb{N}}\subset L_p(\mathbb R^d)$ is a sequence of translates of $\left\lbrace g_k\right\rbrace_{1\le k\le k_0}$ by $\left\lbrace \lambda_i\right\rbrace_{i\in\mathbb{N}}$ and $D$ is a bounded measurable set, then
\begin{align}
\sum\limits_{i=1}^{\infty}\left|\left|f_{i}\big|_D\right|\right|^p<+\infty.\nonumber
\end{align}
\end{lemma}
Now, from Corollary \ref{corolariooperadoralp} and Lemma \ref{lematrasladados3}, we obtain a partial extension of \cite[Proposition 5.1]{Freeman2014} to unconditional approximate Schauder frames.

\begin{proposition}\label{proposicioncompacto}
Let $1< p \le 2$, and let $\left\lbrace f_{i}, f_i' \right\rbrace_{i\in \mathbb{N}}$ be an unconditional approximate Schauder frame of translates for $X\subseteq L_p(\mathbb{R}^d)$ by a uniformly separated sequence.
Then, the restriction operator $R_D:X\rightarrow L_{p}(D)$ given by $R_D(g)=g\big|_D$ is compact for all bounded measurable $D$. Thus, $X\not=L_p(\mathbb{R}^d)$.
\end{proposition}
\begin{proof}
Let $S$ be the frame operator of $\left\lbrace f_{i}, f_i' \right\rbrace_{i\in \mathbb{N}}$, and let $h_i'=\left(S^{-1}\right)^{*}f_i'$ for every $i\in\mathbb{N}$. By Lemma \ref{lemamarcoaproximado}, $\left\lbrace f_{i}, h_i' \right\rbrace_{i\in \mathbb{N}}$ is an unconditional Schauder frame for $X$. By Corollary \ref{corolariooperadoralp}, there is a bounded linear operator $\Psi_p:X\rightarrow \ell_p$ given by
\begin{align}
\Psi_p(g)=&\left(h_i'(g)\right)_{i\in\mathbb{N}}\;\;\;\forall g\in X. \nonumber
\end{align}
For each $n\in\mathbb{N}$, let
\begin{align}
T_n(g)=\sum\limits_{i=1}^{n}h_i'(g)f_i\big|_D. \label{compacto01}
\end{align}
Fix $g\in X$ and $n<m$. Since $p\le p'$, we have
\begin{align}
\left|\left|\left(T_n-T_m\right)(g)\right|\right|=&\left|\left|\sum\limits_{i=n+1}^{m}h_i'(g)f_i\big|_D\right|\right|\le \left(\sum\limits_{i=n+1}^{m}\left|h_i'(g)\right|^{p}\right)^{\frac{1}{p}}\left(\sum\limits_{i=n+1}^{m}\left|\left|f_i\big|_D\right|\right|^{p}\right)^{\frac{1}{p}}\nonumber\\
\le&\left|\left|\Psi_p\right|\right|||g||\left(\sum\limits_{i=n+1}^{\infty}\left|\left|f_i\big|_D\right|\right|^{p}\right)^{\frac{1}{p}}.\label{Cauchy}
\end{align}
It follows from (\ref{compacto01}), (\ref{Cauchy}) and Lemma \ref{lematrasladados3} that $\left\lbrace T_n\right\rbrace_{n\in\mathbb{N}}$ is a Cauchy sequence of finite rank operators. Thus, there is a compact operator $T$ such that $T_n\rightarrow T$ as $n\rightarrow +\infty$. Since $T_n(g)\rightarrow g\big|_D$ as $n\to +\infty$ for every $g\in X$, this completes the proof.
\end{proof}
From Corollary \ref{propositionmarcoincondicionalcud} and Proposition \ref{proposicioncompacto}, we obtain the negative answer of Problem~\ref{pregunta1} for $p=2$ (and a new way to obtain the answer for $1<p<2$).

\begin{corollary}\label{corolariopen12compacto}
Let $1 <p \le 2$. Suppose that $\left\lbrace f_{i}, f_i' \right\rbrace_{i\in \mathbb{N}}$ is an unconditional approximate Schauder frame of translates for $X\subseteq L_p(\mathbb{R}^d)$. If  $\{f_i'\}_{i\in\mathbb{N}}$ is seminormalized, the restriction operator $R_D:X\rightarrow L_{p}(D)$ given by $R_D(g)=g\big|_D$ is compact for all bounded measurable $D\subseteq \mathbb{R}^d$. Hence, $X\not=L_p(\mathbb{R}^d)$.
\end{corollary}

The compactness of the restriction operator in Proposition \ref{proposicioncompacto} gives us more information about the subspace $\overline{\left[f_i:i\in\mathbb{N}\right]}$ for $1<p<2$. In fact, $\overline{\left[f_i:i\in\mathbb{N}\right]}$ is isomorphic to a subspace of $\ell_p$. This follows from an immediate extension of \cite[Proposition 5.3]{Freeman2014} to $L_p(\mathbb{R}^d)$.
\begin{proposition}\label{proposicionembedding}
Let $X$ be a subspace of $L_p(\mathbb{R}^d)$. Suppose that for all cubes $Q\subseteq \mathbb{R}^d$ the restriction operator $R_Q:X\rightarrow L_p(Q)$ given by $R_Q(g)=g\big|_Q$, is compact. Then, $X$ is isomorphic to a subspace of $\ell_p$.
\end{proposition}
The proof is essentially the proof of \cite[Proposition 5.3]{Freeman2014}, with some straightforward modifications to extend it to $L_p(\mathbb{R}^d)$. From Propositions \ref{proposicioncompacto} and \ref{proposicionembedding}, we obtain the following.
\begin{corollary}\label{corolarioembedding}
Let $1< p <2$. If $\left\lbrace f_{i}, f_i' \right\rbrace_{i\in \mathbb{N}}$ is an unconditional approximate Schauder frame of translates for $X\subseteq L_p(\mathbb{R}^d)$ by a relatively uniformly separated sequence, then $X$ is isomorphic to a subspace of $\ell_p$.
\end{corollary}

\subsection{The case $p=1$.}\label{p=1}
It is known that there are no unconditional Schauder frames from $L_1(\mathbb{R}^d)$. This follows from the facts that every space with an unconditional Schauder frame is isomorphic to a complemented subspace of a space with an unconditional Schauder basis \cite[Theorem 3.6]{Casazza1999}, and $L_1(\mathbb{R}^d)$ is not isomorphic to a subspace of a space with an unconditional Schauder basis (see \cite{Pelczynski1961} or Proposition 1.d.1 in \cite{Lindenstrauss1977}). With regard to Schauder frames of translates, it was proven in \cite[Corollary 2.4]{Odell2011} that if $\left\lbrace T_{\lambda_i} f,f_i'\right\rbrace_{i\in\mathbb{N}}$ is a Schauder frame for a subspace $X\subseteq L_1(\mathbb{R})$ and $\left\lbrace \lambda_{i}\right\rbrace_{i\in\mathbb{N}}$ is uniformly separated, then $X$ is isomorphic to a subspace of $\ell_1$. Moreover, it follows from \cite[Proposition~2.1]{Odell2011} that $\left\lbrace T_{\lambda_i} f,f_i'\right\rbrace_{i\in\mathbb{N}}$ satisfies Property ($*$) defined in \cite[Page 6510]{Odell2011}, which in turn implies that the restriction operator $R_I:L_1(\mathbb{R})\rightarrow L_1(I)$ is compact for every bounded interval $I\subseteq \mathbb{R}$. In particular, this ensures that $X\not=L_1(\mathbb{R})$ (this last fact also follows from \cite[Theorem 1.7]{Odell2011} or \cite[Theorem 1]{Bruna2006}).
The described results can be extended in a straightforward manner to $L_1\left(\mathbb{R}^d\right)$ for every $d\in\mathbb{N}$ and to translates of finitely many functions by relatively uniformly separated sequences. Also, for unconditional (approximate) frames, we have the following result.
\begin{proposition}\label{propositionele1}
Let $\left\lbrace f_{i}, f_i' \right\rbrace_{i\in \mathbb{N}}$ be an unconditional approximate Schauder frame of translates for $X\subseteq L_1(\mathbb{R}^d)$ by a  relatively uniformly separated sequence. Then, the operator
\begin{align}
\Psi_1(g)=&\left(f_i'(g)\right)_{i\in\mathbb{N}} \;\;\;\text{for } g\in X \nonumber
\end{align}
is an isomorphism between $X$ and a complemented subspace of $\ell_1$. In particular, $X$ is isomorphic to $\ell_1$.
\end{proposition}
\begin{proof}
Corollary \ref{corolariooperadoralp} allows us to define $\Psi_1:X\rightarrow \ell_1$ as in the statement.
On the other hand, since $\left\lbrace f_i\right\rbrace_{i\in\mathbb{N}}$ is bounded, we can define a bounded linear operator $\Phi_1:\ell_1\rightarrow X$  by
\begin{align}
\Phi_1\left(\textbf{\emph{a}}\right)=&\sum\limits_{i=1}^{\infty}a_if_i \;\;\;\text{for } \textbf{\emph{a}}=\left(a_i\right)_{i\in\mathbb{N}} \in \ell_1.\nonumber
\end{align}
The proof is completed by essentially the same argument given in the proof Proposition~ \ref{proposicionisomorfoaelep}.
\end{proof}
Proposition \ref{propositionele1} can also be easily deduced from the fact that every sequence of translates by a relatively uniformly separated sequence in $L_1\left(\mathbb{R}^d\right)$ can be split into finitely many disjoint subsequences, each of them equivalent to the unit vector basis of $\ell_1$ (see  \cite[Remark 2.8 c)]{Odell2011}).

\subsection{Some related questions.}\label{open}
We end with a list of  questions that arise naturally in the context of (unconditional) Schauder frames of translates.
All the questions make sense for translates of a single function and for translates of finitely many functions are also interesting.

For $p=1$, as we mentioned, an immediate extension of \cite[Corollary 2.4]{Odell2011} shows that there is no Schauder frame $\left\lbrace f_i,f_i'\right\rbrace_{i\in\mathbb{N}}$ for $L_1(\mathbb{R}^d)$ where the $\{f_i\}_{i\in\mathbb{N}}$ are translates of finitely many functions by a relatively uniformly separated sequence in $\mathbb{R}^d$.
It is natural to ask about the existence of Schauder frames of translates with seminormalized coordinates for $L_1(\mathbb{R}^d)$, without any restrictions on the sequence by which a function is translated.
\begin{question}\label{preguntaabiertaL1}
Is there a Schauder frame of translates $\left\lbrace f_i,f_i'\right\rbrace_{i\in\mathbb{N}}$ for $L_1(\mathbb{R}^d)$ with seminormalized $\left\lbrace f_i'\right\rbrace_{i\in\mathbb{N}}$?
\end{question}
Note that if $\left\lbrace f_i,f_i'\right\rbrace_{i\in\mathbb{N}}$ is an approximate Schauder frame for a subspace $X\subseteq L_1(\mathbb{R}^d)$ containing a complemented copy of $\ell_1$ and $\left\lbrace f_i\right\rbrace_{i\in\mathbb{N}}$ is seminormalized, by Corollary \ref{corolarioseminormalizar} there is a seminormalized sequence $\left\lbrace F_i'\right\rbrace\subseteq X'$ such that $\left\lbrace f_i,F_i'\right\rbrace_{i\in\mathbb{N}}$ is a Schauder frame for $X$. Thus, Question \ref{preguntaabiertaL1} is equivalent to the following.
\begin{question}
Is there a Schauder frame of translates $\left\lbrace f_i,f_i'\right\rbrace_{i\in\mathbb{N}}$ for $L_1(\mathbb{R}^d)$?
\end{question}
For  $1<p\le 2$, Proposition \ref{proposicioncompacto} gives a negative answer to Problem \ref{pregunta1} from Section~\ref{sec-intro}. Moreover, it shows that there are no unconditional frames for $L_p(\mathbb{R})$ of the form $\left\lbrace T_{\lambda_{i}}f, f_i' \right\rbrace_{i\in \mathbb{N}}$, for any relatively uniformly separated sequence $\left\lbrace \lambda_{i}\right\rbrace_{i\in\mathbb{N}}$. Again, it is natural to ask what happens if we omit the restrictions on $\left\lbrace \lambda_{i}\right\rbrace_{i\in\mathbb{N}}$.
\begin{question}\label{preguntaabierta1<p<=2}
Let $1<p\le 2$.
\begin{enumerate}[label=\rm{(\alph*)}]
  \item \label{preguntaabierta1<p<=2a} Is there an unconditional Schauder frame of translates $\left\lbrace f_i,f_i'\right\rbrace_{i\in\mathbb{N}}$ for $L_p(\mathbb{R}^d)$?
  \item \label{preguntaabierta1<p<=2b}
Is there a Schauder frame of translates $\left\lbrace f_i,f_i'\right\rbrace_{i\in\mathbb{N}}$ for $L_p(\mathbb{R}^d)$ with seminormalized $\left\lbrace f_i'\right\rbrace_{i\in\mathbb{N}}$?\\

\end{enumerate}
\end{question}

In Theorem \ref{Teoremamarcosincondicionalescon2<p<+infinito}, we showed that for every unbounded sequence $\left\lbrace \lambda_i \right\rbrace_{i\in \mathbb{N}}\subseteq \mathbb{R}^d$, there exists an unconditional Schauder frame with seminormalized coordinates of the form $\left\lbrace T_{\lambda_{m_i}}f, F_i' \right\rbrace_{i\in \mathbb{N}}$, where $f$ is the function constructed in \cite[Theorem 3.2]{Freeman2014}, and $\left\lbrace \lambda_{m_i} \right\rbrace_{i\in \mathbb{N}}$ is a subsequence of $\left\lbrace \lambda_i \right\rbrace_{i\in \mathbb{N}}$. One could ask for necessary and sufficient conditions under which one could keep the original sequence, instead of a (proper) subsequence of it.
\begin{question}\label{preguntaabierta2<p<infinito}
Let $2<p<+\infty$, and let $\left\lbrace \lambda_{i}\right\rbrace_{i\in\mathbb{N}}\subseteq \mathbb{R}^d$ be a relatively uniformly separated sequence. What conditions on $\left\lbrace \lambda_{i}\right\rbrace_{i\in\mathbb{N}}\subseteq \mathbb{R}^d$ ensure the existence of a function $f\in\mathbb{R}^d$ and a seminormalized sequence $\left\lbrace F_i'\right\rbrace_{i\in\mathbb{N}}\subseteq L_{p'}(\mathbb{R}^d)$ such that $\left\lbrace T_{\lambda_{i}}f,F_i'\right\rbrace_{i\in\mathbb{N}}$ is an unconditional Schauder frame for $L_p(\mathbb{R}^d)$?
\end{question}
Finally, we mention that Theorem \ref{Teoremamarcosincondicionalescon2<p<+infinito} only gives a partial answer to Problem~\ref{pregunta1} for the case $2<p<+\infty$.

\subsection*{Ackowledgements} We want to thank Daniel Galicer for useful comments, particularly for the one that motivated our Proposition~\ref{proposicioncomobases2<p<+infinito}. We also want to thank the anonymous referee for carefully reading the manuscript and for helpful and detailed comments and suggestions.

\end{document}